\documentclass[11pt]{article}
\usepackage[latin1]{inputenc}
\usepackage{latexsym}
\usepackage{graphicx}
\usepackage{epsfig}
\usepackage{amssymb}
\usepackage{amsmath}
\usepackage{amsfonts}
\usepackage[english]{babel}   
\usepackage{fancyhdr}
\usepackage{lastpage}
\usepackage{rotating}
\usepackage{tabularx}
\usepackage{algorithm}
\usepackage{multirow}
\usepackage{graphicx}
\usepackage{float}
\usepackage{booktabs}
\usepackage{caption2}
\usepackage{subfigure}
\usepackage{color}
\usepackage{colortbl}
\usepackage{xcolor}
\usepackage{latexsym}
\newtheorem{theorem}{Theorem}
\newtheorem{lemma}{Lemma}

\newcommand{\cqfd}{\mbox{}\nolinebreak\hfill\rule{2mm}{2mm}\medskip\par}

\newenvironment{proof}[1] {\par\noindent{\bf Proof. }{#1}}{\cqfd}
\title{\textbf{On the stability of exact ABCs for the reaction-subdiffusion equation on unbounded domain}\footnote{This research was partially supported by the
National Natural Science Foundation of China under Grant  Nos. 11271173 and 11426174,
 the Research Foundation from the Xi'an university of
Technology under Grant No. 2014CX022, and the Natural Science Basic Research Plan in
Shaanxi Province of China under Grant No.2015JQ1022.}}

\author{
Can Li$^{1}$\footnote{Corresponding author. \,Email addresses:  mathlican@xaut.edu.cn(C.Li)}\\
\\
\small \it    $^1$Department of Applied Mathematics, School of Sciences,\\
\small \it  Xi'an University of Technology, Xi'an, Shaanxi 710054, P.R. China.}
\date{}

\begin{document}
\maketitle
\begin{abstract}
In this note we propose the exact artificial boundary conditions formula to the fractional
reaction-subdiffusion equation on an unbounded domain. With the application of
Laplace transformation, the exact
artificial boundary conditions (ABCs) are derived to
reformulate the original problem on the unbounded domain to an
initial-boundary-value problem on the bounded computational domain. By the  Kreiss theory,
we prove that the reduced initial-boundary value problem is stability.
Based on the properties of tempered fractional calculus, we obtain that
the reduced initial-boundary value problem is long-time stability.
\end{abstract}


\section{Introduction} \label{sec:intro}
In the past two decades, the continuous time random walks (CTRWs) model are popular to describe the  movement of anomalous diffusion particles on the mesoscopic level \cite{Podlubny:99,Metzler:00}. For the sub-diffusion particles, with a long tailed waiting time density and with a reduction in particle concentration deriven by constant per capita linear reaction dynamics,
this process can by described by the following reaction-sub-diffusion equation \cite{Sokolov:06,Henry:06,Henry:08}
\begin{equation}\label{orgproblem}
\frac{\partial u(x,t)}{\partial t}=\kappa_\gamma ~_{0}D_t^{1-\gamma,\lambda}\big(u_{xx}\big)-\lambda u(x,t),x\in \mathbb{R}, \; t>0,
\end{equation}
where $\kappa_\gamma$ is the positive diffusion coefficient, $\lambda$ is a constant reaction rate and  the operator $_{0}D_t^{1-\gamma,\lambda}$ denotes the Riemann-Liouville tempered  fractional derivative \cite{Henry:06,Henry:08,Li:2015}
\begin{equation}\label{TRLD}
_{0}D_t^{\alpha,\lambda}u(t)
=e^{-\lambda t}{{_{0}}D_t^{\alpha}}\big( e^{\lambda t}u(t)\big)=\frac{e^{-\lambda t}}{\Gamma(1-\alpha)}\frac{d}{d
t}\int_{0}^t\frac{e^{\lambda s}u(s)}{(t-s)^{\gamma}}ds,~0<\alpha<1.
\end{equation}

Compare with the previous fractional diffusion models, model \eqref{orgproblem} behaviors many good properties,
such as it preserves the positivity of
solution and  recovers the reaction kinetic equation for homogeneous concentrations, for more physical meaning about this model we refer the references \cite{Sokolov:06,Henry:06,Henry:08}. The model \eqref{orgproblem} is derived in infinite domain.
To get the exact solution of model \eqref{orgproblem}, we usual need the  assumption on boundary
\begin{equation}\label{orgboundary}
u \rightarrow 0,\ |x|\rightarrow \infty.
\end{equation} However, to numerically solve this model, we need impose proper boundaries on the finite computing domain.
One of
powerful tools to overcome the difficulty of unboundedness of the domain  is  the artificial boundary methods (ABMs)\cite{Han:13}.  The main idea of ABMs is to design suitable boundary conditions to absorb the waves arriving at artificial boundaries. And it has been successfully used to solve many kinds of linear and nonlinear partial differential equations with classical derivatives, see the review papers \cite{Antoine:08,Han:13a}.  So far, few work is reported for the fractional diffusion equation on unbounded domain due to the non-locality of fractional operators. In the literature, the exact ABCs are constructed for the one-dimensional time fractional diffusion equation in \cite{Gao12,Gao13}, and the convergence rate and stability of the finite difference methods are established. Brunner et al \cite{Brunner:14} obtained a series of artificial boundary conditions (ABCs) for two-dimensional time fractional diffusion-wave equation. Ghaffari and Hossein \cite{Ghaffari:13} derive the exact ABCs for sub-diffusion equation by using circle artificial boundaries.
Awotunde et al. \cite{Awotunde:15} obtained ABCs for a modified fraction diffusion problem.
Using the similar approach given by Engquist and Majda, Dea obtains a ABCs for the two-dimensional time-fractional wave equation in his recent work \cite{Dea:13}. It is important that the boundaries are designed such that the reduced problem is well-posed. In this note, we focus on the stability of model \eqref{orgproblem} with proper artificial boundary. By the Laplace transformation, we first propose the exact
artificial boundary conditions (ABCs) for model \eqref{orgproblem}. By the  Kreiss theory,
we prove that the reduced initial-boundary value problem is stability.
Base on the properties of tempered fractional calculus, we obtain that
the reduced initial-boundary value problem is long-time stability.

The rest of the note is organized as follows. In Section \ref{eABCsofFDE}, we derive of the exact ABCs for model \eqref{orgproblem}. The original problem  on bounded domain with the boundary condition \eqref{orgboundary} is
 reduced an initial-boundary-value problem on a bounded computational domain.
 In Section \ref{exactstability},
we investigate the stability of the reduced initial-boundary value problem.
\section{The exact artificial boundary conditions}\label{eABCsofFDE}
We first derive the exact artificial boundary conditions for the fractional reaction-subdiffusion equation \eqref{orgproblem} with the boundary condition \eqref{orgboundary}. If we introduce two artificial boundaries
$x_l$ and $x_r,$
Then the real line $\textbf{R}$ is divided  into three parts:
$
\Omega_l :=  \{x|-\infty<x<x_l\}, \quad \Omega_i :=  \{x| x_l<x<x_r\}, \quad \Omega_r :=  \{x| x_r<x<\infty \}.
$
The constants $x_l,\;x_r$ are chosen such that $\sup \{u_0(x)\} \subset \Omega_i$. Denote the exterior domain by $\Omega_e = \Omega_l\cup \Omega_r$. Restricting the solution $u(x,t)$ to $\Omega_e$, we have
\begin{align}
&\frac{\partial u(x,t)}{\partial t}=\kappa_\gamma ~_{0}D_t^{1-\gamma,\lambda}\big(u_{xx}\big)-\lambda u(x,t), & x\in \Omega_{e},t>0,    \label{problemdiffa1}\\
&u(x,0)=0,  & x\in \Omega_{e}, \label{problemdiffa2}\\
& u(x,t)|_{x=x_{l}}= u(x_l,t)\\
&u(x,t)|_{x= x_{r}} = u(x_r,t),     \label{problemdiffa3}\\
&u \rightarrow0,  & \text{as} \ |x|\rightarrow \infty.\label{problemdiffa4}
\end{align}
To derive artificial boundary conditions, we introduce the Laplace transform of a function $f(t)$ by
\begin{equation}\label{deflaplce}
\mathcal {L}\{f(t) ;s\}=\widehat{f}(s)
=\int^{+\infty}_{0}e^{-st}f(t)dt,\quad \text{Re}(s)>0,
\end{equation}
and the inverse Laplace transform of a function $g(s)$ gives

\begin{equation}\label{indeflaplce}
\mathcal {L}^{-1}\{g(s) ;t\}=g(t)
=\frac{1}{2\pi i}\int^{+i\infty}_{-i\infty}e^{st}g(s)ds,\quad i^2=-1.
\end{equation}
 \begin{lemma}
 For $0<\alpha<1$,
the Laplace transform of the  Riemann-Liouville tempered fractional derivative
\cite{Podlubny:99,Henry:06,Henry:08}
\begin{equation}\label{TRLDe}
_{0}D_t^{\alpha,\lambda}f(t)
=\frac{e^{-\lambda t}}{\Gamma(1-\alpha)}\frac{d}{d
t}\int_{0}^t\frac{e^{\lambda s}f(s)}{(t-s)^{\alpha}}ds,
\end{equation}
is given by
  \begin{equation}\label{gLpa}
\mathcal {L}\{ {_{0}D}_t^{\alpha,\lambda}f(t);s\}=(s+\lambda)^{\alpha}\widehat{f}(s).
\end{equation}
And the Laplace transform of the  Caputo tempered fractional derivative \cite{Samko:93}
\begin{equation}\label{TCDe}
      {_{0}^{C}D}_t^{\alpha,\lambda}f(t)=\frac{e^{-\lambda t}}{\Gamma(1-\alpha)}\int_{0}^t\frac{1}{(t-s)^{\alpha}}\frac{d (e^{\lambda s}f(s))}{d
          s}ds,
    \end{equation}
is given by
\begin{equation}\label{caputLpa}
\mathcal {L}\{ {_{0}^{C}D}_t^{\alpha,\lambda}f(t);s\}=(s+\lambda)^{\alpha}\widehat{f}(s)-(s+\lambda)^{\alpha-1}
f(0).
\end{equation}
\end{lemma}
\begin{proof}
The Laplace transform of the  Caputo tempered fractional derivative is directly form \cite{Li:2015}.
For the Laplace transform of the  Riemann-Liouville tempered fractional derivative,
from reference \cite{Li:2015}, we have
  \begin{equation}\label{gLpaa}
\mathcal {L}\{ {_{0}D}_t^{\alpha,\lambda}f(t);s\}=(s+\lambda)^{\alpha}\widehat{f}(s)
-(s+\lambda)\bigg[{_{0}D}_t^{\alpha-1}(e^{\lambda t}f(t))\big|_{t=0}\bigg].
\end{equation}
where $_{0}D_t^{\alpha-1,\lambda}$ denotes  the Riemann-Liouville integral operator  given as
\begin{equation}\label{integralL}
_{0}D_t^{\alpha-1,\lambda}f(t)
=\frac{1}{\Gamma(1-\alpha)}\int_{0}^te^{-\lambda( t-s)}(t-s)^{-\alpha}f(s)ds.
\end{equation}
Furthermore, using the fact
\begin{align}\label{imprelation}
&\lim_{t\rightarrow
0}\frac{1}{\Gamma(1-\alpha)}\int_{0}^te^{-\lambda( t-s)}(t-s)^{-\alpha}f(s)ds\nonumber
\\
&=\lim_{t\rightarrow
0}\frac{e^{-\lambda t}}{\Gamma(1-\alpha)}\int_{0}^t(t-s)^{-\alpha}f(s)e^{\lambda s}ds\nonumber
\\
&=\lim_{t\rightarrow
0}\frac{1}{\Gamma(\alpha)}\bigg[t^{1-\alpha}e^{-\lambda t}
f(0)+\int_{0}^t(t-s)^{1-\alpha}e^{-\lambda (t-s)}[f'(s)+\lambda f(s)]ds\bigg]\\
&=0.\nonumber
\end{align}
 and the assumption on the
smoothness of $f(t)$, we get \eqref{gLpa}.
\end{proof}
Using relation \eqref{gLpa}, and notice that the initial value \eqref{problemdiffa2},
we obtain
\begin{align}
(s+\lambda)^{\gamma}\widehat{u}(x,s)&=\kappa_\gamma \widehat{u}_{xx}(x,s),  \quad x\in\Omega_e,~\text{Re}(s)>0    \label{lpa}\\
\widehat{u} &\rightarrow0,  \quad  \text{as} \ |x|\rightarrow \infty.\label{lpb}
\end{align}
Using the boundary conditions \eqref{lpb}, the solutions of differential equation \eqref{lpa} gives
\begin{align}
\widehat{u}(x,s)&=c_1(s)e^{-\frac{(s+\lambda)^{\gamma/2}}{\sqrt{\kappa_\gamma}} x} ,  \quad x\in \Omega_r, \label{gensoluaa} \\
\widehat{u}(x,s)&=c_2(s)e^{\frac{(s+\lambda)^{\gamma/2}}{\sqrt{\kappa_\gamma}} x},   ~~\quad x\in \Omega_l.\label{gensoluba}
\end{align}
Differentiating equations \eqref{gensoluaa} and  \eqref{gensoluba} with respect to variable $x$, we have
\begin{align}
\widehat{u}_x(x,s)&=-\frac{(s+\lambda)^{\gamma/2}}{\sqrt{\kappa_\gamma}} \widehat{u}(x,s),  \quad x\in\Omega_r,  \label{dgensolua}  \\
\widehat{u}_x(x,s)&=\frac{(s+\lambda)^{\gamma/2}}{\sqrt{\kappa_\gamma}} \widehat{u}(x,s),  ~~\quad x\in \Omega_l.\label{dgensolub}
\end{align}
Remembering the  Laplace transform formula \eqref{caputLpa}
 and taking the inverse Laplace transform to \eqref{dgensolua}-\eqref{dgensolub} , give as
\begin{align}
u_{x}(x,t)&=-\frac{1}{\sqrt{\kappa_\gamma}}\frac{e^{-\lambda\gamma}}{\Gamma(1-\frac{\gamma}{2})}\int_{0}^{t}\frac{e^{\lambda\gamma}u_\tau (x,\tau) }{(t-\tau)^{\frac{\gamma}{2}}}d \tau  , \quad x=x_r,    \label{inversedgensolua}\\
u_{x}(x,t)&=\frac{1}{\sqrt{\kappa_\gamma}}\frac{e^{-\lambda\gamma}}{\Gamma(1-\frac{\gamma}{2})}\int_{0}^{t}\frac{e^{\lambda\gamma}u_\tau (x,\tau) }{(t-\tau)^{\frac{\gamma}{2}}}d \tau,  \quad x=x_l.\label{inversedgensolub}
\end{align}
which are equivalent to
\begin{align}
u_{x}(x,t)&= -\frac{1}{\sqrt{\kappa_\gamma}}~{_{0}^{C}\!D}_{t}^{\frac{\gamma}{2},\lambda}u(x,t),  \quad x=x_r,    \label{inversedgensoluav}\\
u_{x}(x,t)&= \frac{1}{\sqrt{\kappa_\gamma}}~{_{0}^{C}\!D}_{t}^{\frac{\gamma}{2},\lambda}u(x,t),  ~~\quad x=x_l.\label{inversedgensolubv}
\end{align}
where we used the Laplace transform of the Caputo fractional derivative \eqref{caputLpa} and  the initial value condition \eqref{problemdiffa2}.
If take \eqref{inversedgensolua}-\eqref{inversedgensolub} as artificial boundary conditions, then the previous problem
\eqref{problemdiffa1}-\eqref{problemdiffa4} on unbounded domain is reduced to the following initial-boundary value problem on bounded domain
\begin{align}
&\frac{\partial u(x,t)}{\partial t}=\kappa_\gamma ~_{0}D_t^{1-\gamma,\lambda}\big(u_{xx}\big)-\lambda u(x,t),   &x\in \Omega_{in}, \;t>0,   \label{exacta1}\\
&u(x,0)=u_0(x), & x\in\Omega_i ,\label{exacta2}\\
& u_x(x,t)=\frac{1}{\sqrt{\kappa_\gamma}}~{_{0}^{C}D}_{t}^{\frac{\gamma}{2},\lambda}u(x,t), & x=x_l,\\
& u_x(x,t)=-\frac{1}{\sqrt{\kappa_\gamma}}~{_{0}^{C}D}_{t}^{\frac{\gamma}{2},\lambda}u(x,t), &  x=x_r.\label{exacta3}
\end{align}
\section{Stability analysis of the reduced problems} \label{exactstability}
Next we consider the stability analysis of the reduced problem \eqref{exacta1}-\eqref{exacta3}.
The main tool is the Kreoss theory \cite{KL89}. Firstly, we introduce the notations of the inner product, the classic $L^2$ norm  given by
\begin{align*}
 (u,v)=\int_{\Omega_i}u(x)v(x)dx, \quad \|u\|^2_{L^2(\Omega_{i})}=\|u\|^2=(u,u).
\end{align*}

\begin{lemma}[\cite{Yan:92}]\label{lema2p}(Parseval relation)Suppose that $u(t),v(t)$ is well defined on
$[0,+\infty)$, then for $s_0>0$ we have
\begin{equation}\label{prelationp}
\int_{-\infty}^{+\infty}\hat{u}(s_0+i\zeta)\hat{v}(s_0-i\zeta)d\zeta=
2\pi\int_{0}^{+\infty}e^{-2s_0t}u(t)v(t)dt,i^2=-1.
\end{equation}
\end{lemma}
\begin{lemma}[\cite{Yan:92}]\label{lema2in}Suppose that $v(t)$ is well defined on
$[0,\infty)$, then for $s_0>0$ we have
\begin{equation}\label{2.17}
\int_{0}^{+\infty}e^{-2s_0 t}\frac{d v(t)}{dt}v(t)dt\leq0.
\end{equation}then
\begin{equation}\label{2.18}
\int_{0}^{+\infty}e^{-2s_0 t}v^2(t)dt\leq \frac{1}{2s_0}v^2(0).
\end{equation}
\end{lemma}

\begin{theorem}\label{stabilityLABC}
The initial-boundary value  problem \eqref{exacta1}-\eqref{exacta3} holds the prior estimate
\begin{align} \label{IQ1}
 \|u\|_{L^2(H^{2,s})}^2\leq \frac{1}{2s_0}\|u_0(x)\|^2,
\end{align}
where $\|u\|_{L^2(H^{2,s})}=\int_{0}^{+\infty}\|u(\cdot,t)\|^2e^{-2s_0t}dt$.
\end{theorem}
\begin{proof}
Applying the Laplace transform to \eqref{exacta1}-\eqref{exacta3}, we have
\begin{align}
 \frac{\partial \widehat{u}(x,s)}{\partial t}&=\kappa_\gamma(s+\lambda)^{1-\gamma}\widehat{u}_{xx}(x,s)-\lambda
    \widehat{u}(x,s),\quad x\in \Omega_{in},    \label{Lerrexacta1}\\
u(x,0)&=u_0(x), \quad x\in\Omega_i ,\label{Lerrexacta2}\\
\widehat{u}_x(x_l,s)&=\frac{(s+\lambda)^{\gamma/2}}{\sqrt{\kappa_\gamma}} \widehat{u}(x_l,s),\\
\widehat{u}_x(x_r,s)&=-\frac{(s+\lambda)^{\gamma/2}}{\sqrt{\kappa_\gamma}} \widehat{u}(x_r,s), \label{Lerrexacta3}
\end{align}
where $\frac{\partial \widehat{u}(x,s)}{\partial t}:=s\widehat{u}(x,s)-u_0(x)$.
Multiplying  equation \eqref{Lerrexacta1} by $\overline{\widehat{u}(x,s)}$ and integrate on $\Omega_{in}$, we obtain
\begin{equation}\label{integralB}
    \bigg(\frac{\partial\widehat{u}}{\partial t},\widehat{u}\bigg)
    =\kappa_\gamma(s+\lambda)^{1-\gamma}\big(\widehat{u}_{xx}(x,s),\widetilde{u}(x,s)\big)-\lambda
    \big(\widehat{u},\widehat{u}\big).
\end{equation}
By integrating \eqref{integralB} by parts we get
\begin{equation}\label{integralA}
   \bigg(\frac{\partial\widehat{u}}{\partial t},\widehat{u}\bigg)=-\kappa_\gamma(s+\lambda)^{1-\gamma}\|\widehat{u}_{x}\|^{2}-\kappa_\gamma(s+\lambda)^{1-\gamma}\widehat{u}_{x}(x,s)\overline{\widehat{u}(x,s)}\big|^{x_r}_{x_l}-\lambda
    \|\widetilde{u}\|^{2}.
\end{equation}
In view of the boundary conditions \eqref{Lerrexacta3}, we have
\begin{align}\label{integralC}
\bigg(\frac{\partial\widehat{u}}{\partial t},\widehat{u}\bigg)=&-\kappa_\gamma (s+\lambda)^{1-\gamma}\|\widehat{u}_{x}\|^{2}-\lambda
 \|\widetilde{u}\|^{2} \nonumber\\
 &-\sqrt{\kappa_\gamma}(s+\lambda)^{1-\frac{\gamma}{2}}\big[|\widehat{u}(x_l,s)|^2+|\widehat{u}(x_r,s)|^2\big].
\end{align}
Hence, we have
\begin{align}\label{integralC}
   Re\bigg(\bigg(\frac{\partial\widehat{u}}{\partial t},\widehat{u}\bigg)\bigg)=&-\kappa_\gamma Re\big((s+\lambda)^{1-\gamma}\big)\|\widehat{u}_{x}\|^{2}-\lambda
    Re\big(\|\widetilde{u}\|^{2}\big)\\
   &-\sqrt{\kappa_\gamma}Re\big((s+\lambda)^{1-\frac{\gamma}{2}}\big)\big[|\widehat{u}(x_l,s)|^2+|\widehat{u}(x_r,s)|^2\big].
\end{align}
Using the fact $\arg(s_0+\lambda+i\zeta)\in(-\pi/2,\pi/2]$ for $s_0+\lambda>0$, we have
\begin{align}\label{postiveker}
   (s_0+\lambda+i\zeta)^{\beta}&=\big(|(s_0+\lambda+i\zeta)|e^{i\arg(s_0+\lambda+i\zeta)}\big)^{\beta}\\
   &=\big((s_0+\lambda)^2+(\zeta)^2\big)^{\frac{\beta}{2}}e^{\beta i\arg((s_0+\lambda+i\zeta))}.
\end{align}
which means $Re((s+\lambda)^{\beta})=|(s+\lambda)|^{\beta}\cos(\beta \arg((s+\lambda)))>0,$ for $ \beta=1-\gamma\textrm{or}~1-\frac{\gamma}{2},\gamma\in(0,1)$.
Above inequality implies $Re\big((s+\lambda)^{1-\gamma}\big)>0$ and $~Re\big((s+\lambda)^{1-\frac{\gamma}{2}}\big)>0,$ and hence
$
Re\bigg(\big(\frac{\partial\widehat{u}}{\partial t},\widehat{u}\big)\bigg)\leq0.
$
Paseval's relation \eqref{prelationp} then leads
\begin{align}\label{integralDh}
    \int_{-\infty}^{+\infty}\bigg(\frac{\partial\widehat{u}}{\partial t},\widehat{u}\bigg)(s_0+i\zeta)d\zeta
    =&\int_{\Omega_{i}}\int_{-\infty}^{+\infty}\frac{\partial\widehat{u}}{\partial t}(s_0+i\zeta)\overline{\widehat{u}}(s_0+i\zeta)d\zeta dx \nonumber\\
    =&2\pi\int_{\Omega_{i}}\int_{0}^{+\infty}e^{-2s_0t}\frac{\partial u}{\partial t}u(t)dtdx.
\end{align}
Furthermore, using Lemma \ref{lema2in}, we have
\begin{equation}\label{Hstability}
\int_{\Omega_{i}}\int_{0}^{+\infty}e^{-2s_0t}\frac{\partial u}{\partial t}u(x,t)dtdx\leq \frac{1}{2s_0}\int_{\Omega_{i}}u^2_0(x)dx,
\end{equation}
which is the desired inequality \eqref{IQ1}.
\end{proof}
The model \eqref{orgproblem} can be rewritten  as
\begin{equation}\label{orgproblema}
\frac{\partial e^{\lambda t}u(x,t)}{\partial t}=\kappa_\gamma ~_{0}D_t^{1-\gamma}\big(e^{\lambda t}u_{xx}\big).
\end{equation}
Performing Riemann-Liouville fractional integral operator $~_{0}D_t^{\gamma-1}$ on both side of \eqref{orgproblema},
 using the composite properties  fractional derivative and integral \cite{Podlubny:99}, we arrive at
\begin{equation}\label{orgproblemb}
{_{0}^{C}D}_t^{\gamma}\big(e^{\lambda t}u(x,t)\big)=\kappa_\gamma e^{\lambda t}u_{xx},
\end{equation}
or, equivalently,
\begin{equation}\label{orgproblemc}
{_{0}^{C}D}_t^{\gamma,\lambda}(u(x,t))=\kappa_\gamma u_{xx},
\end{equation}
Hence, the reduced problem  \eqref{exacta1}-\eqref{exacta3} equivalents to  the following initial-boundary value problem
\begin{align}
&{_{0}^{C}D}_t^{\gamma}\big(e^{\lambda t}u(x,t)\big)=\kappa_\gamma \big(e^{\lambda t}u(x,t)\big)_{xx},   &x\in \Omega_{i}, \;t>0,   \label{exacta1eq}\\
&u(x,0)=u_0(x), & x\in\Omega_i ,\label{exacta2eq}\\
& \big(e^{\lambda t}u(x,t)\big)_{x}=\frac{1}{\sqrt{\kappa_\gamma}}~{_{0}^{C}D}_{t}^{\frac{\gamma}{2}}\big(e^{\lambda t}u(x,t)\big), & x=x_l,\label{exacta3eqa}\\
& \big(e^{\lambda t}u(x,t)\big)_{x}=-\frac{1}{\sqrt{\kappa_\gamma}}~{_{0}^{C}D}_{t}^{\frac{\gamma}{2}}\big(e^{\lambda t}u(x,t)\big), &  x=x_r.\label{exacta3eqb}
\end{align}
For the initial-boundary value problem \eqref{exacta1eq}-\eqref{exacta3eqb}, we have the long-time stability.
\begin{theorem}\label{stabilityLABCL2}The solution $u(x, t)$ of the initial-boundary value  problem \eqref{exacta1eq}-\eqref{exacta3eqb} holds the prior estimate
\begin{align} \label{IQ1thm}
 \|u(x,t)\|^2+2\kappa_\gamma~{_{0}D}_t^{-\gamma,2\lambda}\|u_x(x,t)\|^{2}
 \leq e^{-2\lambda t}\|u_0(x)\|^{2}.
\end{align}where ${_{0}D}_t^{-\gamma,2\lambda}$ denotes  the Riemann-Liouville integral operator given in \eqref{integralL}.
\end{theorem}
\begin{proof}
Taking $v(x,t)=e^{\lambda t}u(x,t)$, form \eqref{exacta1eq}-\eqref{exacta3eqb}, we have
\begin{align}
&{_{0}^{C}D}_t^{\gamma}v(x,t)=\kappa_\gamma v_{xx}(x,t),   &x\in \Omega_{i}, \;t>0,   \label{exacta1eqiq}\\
&u(x,0)=u_0(x), & x\in\Omega_i ,\label{exacta2eqiq}\\
&v_x(x,t)=\frac{1}{\sqrt{\kappa_\gamma}}~{_{0}^{C}D}_{t}^{\frac{\gamma}{2}}v(x,t), & x=x_l,\label{exacta3eqiqa}\\
&v_x(x,t)=-\frac{1}{\sqrt{\kappa_\gamma}}~{_{0}^{C}D}_{t}^{\frac{\gamma}{2}}v(x,t), &  x=x_r.\label{exacta3eqiqb}
\end{align}
We multiply equation \eqref{exacta1eqiq} by $v(x,t)$ and integrate on $\Omega_{in}$, we get
\begin{equation}\label{integralAiq}
   \big({_{0}^{C}D}_t^{\gamma}v,v\big)=-\kappa_\gamma\|v_x\|^{2}+\kappa_\gamma v_x(x,t) v(x,t)\big|^{x_r}_{x_l}.
\end{equation}
In view of the boundary conditions \eqref{exacta3eqiqa}-\eqref{exacta3eqiqb}, we have
\begin{align}\label{integralCiq}
\big({_{0}^{C}D}_t^{\gamma}v,v\big)=-\kappa_\gamma\|v_x\|^{2}-\sqrt{\kappa_\gamma}~\big[v(x_r,t){_{0}^{C}D}_{t}^{\frac{\gamma}{2}}v(x_r,t) +v(x_l,t){_{0}^{C}D}_{t}^{\frac{\gamma}{2}}v(x_l,t) \big].
\end{align}
With help of the inequality \cite{Alikhanov:01}
\begin{equation}\label{Caputoinequality}
w(t){_{0}^{C}D}_t^{\alpha}w(t)\geq\frac{1}{2}{_{0}^{C}D}_t^{\alpha}(w^2(t)),0<\alpha<1,
\end{equation}
and the fact $${_{0}^{C}D}_{t}^{\frac{\gamma}{2}}~\big[v^2(x_r,t)+v^2(x_l,t)\big]\geq0,$$
we get
\begin{align}\label{fiintegralCiq}
{_{0}^{C}D}_t^{\gamma}\|v\|^2+2\kappa_\gamma\|v_x\|^{2}\leq 0.
\end{align}

By applying the fractional integral  operator ${_{0}D}_t^{-\gamma}$ to both sides of inequality \eqref{fiintegralCiq},
using the fact \cite{Li:07}
$${_{0}D}_t^{-\gamma}~{_{0}D}_t^{\gamma}(w(t))=w(t)-w(0),$$
 we
obtain the estimate
\begin{align}\label{finalest}
\|v\|^2+2\kappa_\gamma~{_{0}^{C}D}_t^{-\gamma}\|v_x\|^{2}\leq \|v_0(x)\|^2.
\end{align}
Taking $u(x,t)=e^{-\lambda t}v(x,t)$ we get  the estimate \eqref{IQ1thm}.
\end{proof}
\section*{Acknowledgements}
 The author would like to thank Dr. J.W.Zhang for his very helpful comments and suggestions.


\end{document}